\documentclass[11pt,reqno]{amsart}

\usepackage[T1]{fontenc}
\usepackage{lmodern}
\usepackage{amsmath,amssymb,amsthm}
\usepackage[letterpaper,margin=1in]{geometry}
\usepackage{microtype}
\usepackage[hidelinks]{hyperref}
\usepackage{enumitem}

\usepackage{mathtools}
\mathtoolsset{showonlyrefs=true}

\hypersetup{
  pdftitle={Complexity of output feedback stabilization},
  pdfsubject={Strong NP-hardness in continuous and discrete time and exponential encoding length of stabilizing controllers},
  pdfkeywords={output feedback stabilization, Hurwitz stability, strong NP-hardness, 3SAT}
}

\newtheorem{theorem}{Theorem}
\newtheorem{lemma}[theorem]{Lemma}

\theoremstyle{remark}

\theoremstyle{plain}
\newcommand{\R}{\mathbb R}
\newcommand{\C}{\mathbb C}
\newcommand{\Z}{\mathbb Z}
\newcommand{\Q}{\mathbb Q}

\DeclareMathOperator{\diag}{Diag}
\DeclareMathOperator{\Ree}{Re}

\newcommand{\norm}[1]{\left\lVert#1\right\rVert}
\newcommand{\abs}[1]{\left\lvert#1\right\rvert}

\newcommand{\T}{\mathsf T}
\newcommand{\sym}{\operatorname{sym}}
\newcommand{\SAT}{\textnormal{\textsc{3SAT}}}
\newcommand{\M}{\mathcal M}
\newcommand{\G}{\mathcal G}
\newcommand{\calC}{\mathcal C}
\allowdisplaybreaks[1]

\usepackage{xcolor}

\long\def\ac#1{{\color{black}#1}}

\title[Output feedback stabilization]{Complexity of output feedback stabilization}

\author[Amir Ali Ahmadi]{Amir Ali Ahmadi$^{*}$}
\author[Abraar Chaudhry]{Abraar Chaudhry$^{\dagger}$}
\author[Ijay Narang]{Ijay Narang$^{\ddagger}$}
\author[Yukai Tang]{Yukai Tang$^{*}$}

\thanks{$^{*}$Princeton University, Department of Operations
Research and Financial Engineering. Partially supported by the Princeton AI Lab Seed Grant, the Princeton SEAS Innovation Grant, and a Research Gift in Mathematical Optimization. 
Email: \texttt{aaa@princeton.edu}, \texttt{yt3846@princeton.edu}.}

\thanks{$^{\dagger}$University of Colorado Boulder,
Department of Applied Math.
Email: \texttt{abraar.chaudhry@colorado.edu}.}

\thanks{$^{\ddagger}$Georgia Institute of Technology,
School of Computer Science.
Email: \texttt{inarang3@gatech.edu}.}

\date{}
\begin{document}

\begin{abstract}
We show that unless $\mathrm P=\mathrm{NP}$, there cannot be a polynomial-time
(or even pseudo-polynomial-time) algorithm for output feedback
stabilization of a linear dynamical system with a linear controller. This settles one of the best-known open problems in control theory. The result holds in both continuous and discrete time. We also present a family of stabilizable linear dynamical systems for which no polynomial-time algorithm can write down a stabilizing controller in its standard representation. 

\end{abstract}

\maketitle

\emph{AI Declaration:}
On August 2, 2026, the authors realized that ChatGPT 5.6 could produce a one-shot proof of NP-hardness of output feedback stabilization. However, this proof and its initial variants were long and hard to verify, relied on techniques unfamiliar to us, and offered little intuition. We have since worked hand in hand with this tool to simplify the proof, make it more interpretable, and base it on tools familiar to most control theorists. This involved proposing plausible proof techniques, checking the resulting arguments, and replacing some with our own. The current proof gives a reduction directly from 3SAT, making the combinatorial source of the hardness clear. It also relies on nothing but basic matrix analysis and linear systems theory. The same general approach was applied to all results in this paper and the authors assume responsibility for all content. We believe that, given the remarkable pace of AI's progress in automated theorem proving, an emerging role for mathematicians is to teach AI which proofs come closer to being, as Erd\H{o}s put it, from ``The Book.''


\section{Introduction}


A real matrix is \emph{Hurwitz} if all its eigenvalues have negative real
parts, and \emph{Schur} if all its eigenvalues lie in the open
unit disk in the complex plane. For matrices $A\in\Q^{N\times N}$, $B\in\Q^{N\times p}$, and
$C\in\Q^{q\times N}$, the \emph{output feedback stabilization} problem in continuous time (resp. discrete time) asks whether
there exists some $K\in\R^{p\times q}$ that makes $A+BKC$ Hurwitz (resp.
Schur). This amounts to asymptotically stabilizing the linear dynamical system
$\dot x=Ax+Bu$, $y=Cx$, or its discrete-time analogue, using a static linear
control law $u=Ky$ that maps the observed output $y$ of the system to the control $u$. When the answer to this question is positive, one often also hopes to find a stabilizing ``feedback gain'' matrix $K$.

In their 1995 survey, Blondel, Gevers, and Lindquist \cite{BlondelGeversLindquist}
described output feedback stabilization as ``the most often mentioned
specific open problem'' in systems and control theory. When $C$ is the identity matrix, the problem reduces to \emph{state feedback stabilization} and admits
an efficient test based on the stabilizability of the pair $(A,B)$~\cite{Wonham}, or a polynomial-size semidefinite programming formulation that can also recover the feedback gain matrix~\cite[Section~7.2.1]{BoydEtAl}. In general, one can reformulate the output feedback stabilization problem as the problem of testing the feasibility of a system of polynomial inequalities\footnote{Such a formulation can be obtained, for example, by jointly searching for the controller matrix $K$ and a quadratic Lyapunov function for the closed-loop system, and then invoking Sylvester's criterion for matrix positive definiteness.} and hence the problem admits exponential-time algorithms; see, e.g.,~\cite{grigor1988solving}. However, no polynomial-time algorithm is known for output feedback stabilization. Blondel and Tsitsiklis \cite{BlondelTsitsiklis} proved NP-hardness when
the entries of $K$ are constrained to prescribed intervals, but left
the complexity of the unrestricted problem open. Chaudhry
\cite{Chaudhry} proved strong NP-hardness of the more general problem
of deciding whether an affine subspace of matrices contains a Hurwitz
(or a Schur) matrix. The affine subspaces in that reduction, however, do
not have the specific form $\{A+BKC:K\in\R^{p\times q}\}$ required
by output feedback stabilization.

In this paper, we prove that output feedback stabilization is strongly NP-hard. This result is established both in continuous time (Section~\ref{sec:ct}) and in discrete time (Section~\ref{sec:dt}). The implication of this result is that unless P=NP, there is no polynomial-time (or even pseudo-polynomial-time) algorithm for output feedback stabilization. In Section~\ref{sec:controller-size}, we also establish, in both continuous and discrete time, a representation obstruction to polynomial-time algorithms that is independent of any complexity-theoretic assumption: even when the system is stabilizable, a polynomial-time algorithm cannot, in general, write down a stabilizing feedback gain matrix in its natural representation.



All results in this paper are in the Turing model of computation, where
the input consists of the rational entries of the matrices $A,B,C$.
Each entry is represented as $a/b$, where $a$ is a signed integer and $b$ is a positive integer. The input size $L$ is the total number of bits needed to encode these numerators and denominators in binary, including the signs. Let $M$ be the maximum of $|a|$ and $b$ over all entries of $A,B,C$.
An algorithm is polynomial-time if its running time is
bounded by a polynomial in $L$, and pseudo-polynomial-time if its running time is bounded by a polynomial in $M$ and the dimensions of $A,B,C$. In particular, every polynomial-time algorithm is also pseudo-polynomial-time, though the converse is not true as pseudo-polynomial time algorithms can take exponential time in $L$. Our reductions establish NP-hardness even when $M$ is polynomially bounded in the matrix dimensions. On these instances, any pseudo-polynomial-time bound becomes a polynomial-time bound.
Thus, our NP-hardness results are in the strong sense and rule out pseudo-polynomial-time algorithms unless $\mathrm{P}=\mathrm{NP}$. See~\cite{GareyJohnson} for more details on these notions.

\section{Continuous-time hardness}\label{sec:ct}

In this section, we prove the following theorem.

\begin{theorem}\label{thm:main}
Given matrices
$A\in\Z^{N\times N}$, $B\in\Z^{N\times p}$, and $C\in\Z^{1\times N}$, deciding if there exists a matrix $K\in\R^{p\times1}$ that makes $A+BKC$ Hurwitz is
strongly NP-hard.
\end{theorem}

Note that the theorem establishes strong NP-hardness of output feedback stabilization even when the system has a single output. The result trivially extends to systems with any fixed number of outputs. The proof will further show the claim even when $C$ is binary, every entry of $A$ has absolute value at most $256N^2$, and every entry of $B$ has absolute value at most $6$.

We shall prove Theorem~\ref{thm:main} by a reduction from 3SAT which proceeds in two steps. In Step~1, which is the easier step, we reduce 3SAT to the problem of simultaneous output feedback stabilization of multiple dynamical systems by a common controller. In Step~2, which is more involved, we reduce these simultaneous-stabilization instances to output feedback stabilization of a single system. Since the closed-loop system in all cases is linear, its asymptotic stability can be shown with algebraic certificates involving quadratic forms. This allows Step 2 to be carried out by relating the algebraic certificates for the two stabilization problems. 
%
We note that (weak) NP-hardness of simultaneous output feedback stabilization, corresponding to Step 1, was previously shown by Blondel and Tsitsiklis~\cite{BlondelTsitsiklis}. However, we need a different reduction both to establish strong NP-hardness, and because the construction used in Step 2 is closely intertwined with the particular simultaneous-stabilization instances produced in Step 1.






To present the proof formally, we begin by introducing some notation and establishing three technical lemmas.\footnote{The proof of Step 1 relies on Lemma~\ref{lem:certificate} and Lemma~\ref{lem:property-encoding}, while Step 2 relies on Lemma~\ref{lem:certificate}, Lemma~\ref{lem:completion}, and Lemma~\ref{lem:property-encoding}.} For a real matrix $E$, we write its transpose as $E^\T$ and $\sym(E)=(E+E^{\T})/2$. We denote its kernel by $\ker E$.
We use the notation
$\diag$ for block-diagonal assembly, $I_n$ for the $n\times n$ identity matrix, and $\norm{\cdot}$ for the Euclidean vector norm or its induced matrix norm. For a symmetric matrix $X$, the notation $X\succeq0$ (resp. $X\succ0$) denotes that $X$ is positive semidefinite (resp. positive definite); i.e., it has nonnegative (resp. positive) eigenvalues. For a positive integer $m$, denote the set $\{1,\ldots,m\}$ by $[m]$. For a complex number $z\in \C$, we denote its real part by $\Ree(z)$ and its conjugate by $\bar{z}$. For a complex matrix $E$, we denote its conjugate transpose by $E^*$. We use the notation $\vee$ for Boolean OR and $\wedge$ for Boolean AND.


Let us define
\(
J=\begin{pmatrix}0&0\\1&0\end{pmatrix}, e=\begin{pmatrix}0\\1\end{pmatrix}.
\)
For real scalars $\lambda_i$ and vectors $u_i,v_i\in\R^2$, $i=1,\ldots,m$, we define $D_i=\lambda_iI_2-J$, and
\begin{equation}\label{eq:stacks}
 D=\diag(D_1,\ldots,D_m)\in\R^{2m\times2m},\quad
 u=\begin{pmatrix}u_1\\\vdots\\u_m\end{pmatrix},\quad
 v=\begin{pmatrix}v_1\\\vdots\\v_m\end{pmatrix},\quad
 c=\begin{pmatrix}e\\\vdots\\e\end{pmatrix}\in\R^{2m}.
\end{equation}
Define
$\M\in\R^{4m\times4m}$ and $\calC\in\R^{1\times4m}$ by
\begin{equation}\label{eq:realization}
 \M=\begin{pmatrix}0&-D-uc^{\T}\\I_{2m}&-vc^{\T}\end{pmatrix},
 \qquad \calC=(0_{1\times2m},c^{\T}).
\end{equation}

The next lemma relates Hurwitzness of $\M$ to Lyapunov-type positive (semi)definiteness conditions on symmetric matrices.

\begin{lemma}[Hurwitzness of $\M$]\label{lem:certificate}
Let $m$ be a positive integer, $\lambda_1,\ldots,\lambda_m$ be distinct real scalars, and $u_i,v_i\in \R^2$ for $i\in[m]$. Let $D,u,v,c$ be defined as in~\eqref{eq:stacks} and $\M$ be defined as in~\eqref{eq:realization}.
Suppose $X=X^{\T}\in\R^{2m\times2m}$ satisfies
\begin{equation}\label{eq:completion}
 Xc=v,\qquad DX-XD^{\T}=vu^{\T}-uv^{\T}.
\end{equation}
Then the following claims hold:
\begin{enumerate}[label=\textup{(\roman*)},leftmargin=*,topsep=3pt,itemsep=2pt]
\item If $\M$ is Hurwitz, then $X\succeq0$.
\item If $X\succ0$ and $\sym(D+uc^{\T})\succ0$, then $\M$ is Hurwitz. 
\end{enumerate}
\end{lemma}
\begin{proof}
We first prove (i). Let $E=D+uc^{\T}$, $Y=DX+uv^{\T}$, and $\calC$ be defined as in~\eqref{eq:realization}.
The second equation in~\eqref{eq:completion} gives $Y=Y^{\T}=EX$. Define
\(
H=\diag(Y,X).
\)
Since $Xc=v$, we observe that the matrix
\begin{equation}\label{eq:lyapunov-identity}
    \M H+H\M^{\T}
= -2(H\calC^\T)(H\calC^\T)^\T
\end{equation}
is negative semidefinite.
Suppose for the sake of contradiction that $X\not\succeq0$. Since $X$ is a principal submatrix of $H$, this implies $H \not\succeq0$. 
Hence, there exists
$z\in\R^{4m}$ and $\epsilon > 0$ such that $z^{\T}Hz\leq -\epsilon$. Along the trajectory of the system $\dot x=\M^{\T}x$
starting at $x(0)=z$, the quadratic form $V(x)=x^{\T}Hx$ satisfies
\[
\dot V(x)=x^{\T}(\M H+H\M^{\T})x.
\]
Since $\M H+H\M^{\T}$ is negative semidefinite, \(\dot V(x)\leq 0\). Thus, $V(x(t))\le V(z)\leq -\epsilon$ for all $t\ge0$. However, $\M^\T$ being Hurwitz implies that
$x(t)\to0$, hence $V(x(t))\to0$, which is a contradiction, proving~\textup{(i)}.


We now prove~\textup{(ii)}. Assume $X\succ0$ and $E+E^{\T}\succ0$. 
We first show that $Y\succ 0$. Suppose for the sake of contradiction that $Y\not\succ 0$. Since $Y$ is symmetric, there is a nonzero vector $z$ and $\lambda \leq 0$ such that $Yz = \lambda z$. Then, we have
\[
 z^\T (X^{-1}Y+YX^{-1}) z = 2\lambda z^\T X^{-1}z \leq 0. 
\]
This is a contradiction since 
\[
X^{-1}Y+YX^{-1}=E^{\T}+E\succ0.
\] 
Therefore, $Y\succ 0$ and
\(
H=\diag(Y,X)\succ0.
\)
We next show that no eigenvector of $\M$ lies in $\ker\calC$.
Suppose, for contradiction, that there exist $\mu\in\C$ and
$(p^\T,q^\T)^\T\in\C^{4m}\setminus\{0\}$ such that
\[
\M\begin{pmatrix}p\\q\end{pmatrix}
=
\mu\begin{pmatrix}p\\q\end{pmatrix},
\qquad
\calC\begin{pmatrix}p\\q\end{pmatrix}
=
c^\T q=0.
\]
By the definition of $\M$, the eigenvalue equation gives
\[
-(D+uc^\T)q=\mu p,
\qquad
p-vc^\T q=\mu q.
\]
Since $c^\T q=0$, these equations reduce to
\(
-Dq=\mu p,\
p=\mu q,
\)
and hence
\(
Dq=-\mu^2q.
\)
Write $q=(q_1^\T,\ldots,q_m^\T)^\T$, with $q_j\in\C^2$.
Since $D_j=\lambda_jI_2-J$, for every $j\in[m]$ we have
\begin{equation}\label{eq:preceding-equation}
    \bigl((\lambda_j+\mu^2)I_2-J\bigr)q_j=0.
\end{equation}
If $\lambda_j+\mu^2\neq0$, then since $J$ is nilpotent, we have $q_j = 0$. Since the $\lambda_j$ are distinct, at most one block
$q_i$ can be nonzero. Since $(p^\T,q^\T)^\T\neq0$, there is $i\in[m]$ such that
$\lambda_i+\mu^2=0$. Equation~\eqref{eq:preceding-equation} gives $Jq_i=0$.
As $\ker J=\operatorname{span}\{e\}$, we have $q_i=\alpha e$ for
some $\alpha\in\C$. But then
\[
0=c^\T q=e^\T q_i=\alpha,
\]
so $q_i=0$. Thus $q=0$, and consequently $p=\mu q=0$, contradicting
$(p^\T,q^\T)^\T\neq0$. Therefore every nonzero eigenvector $z$ of $\M$ satisfies
$\calC z\neq0$. 

Let $\mu\in\C$ be an eigenvalue of $\M$, and let
$z\in\C^{4m}\setminus\{0\}$ be a corresponding eigenvector. We have
\[
-2z^*\calC^{\T}\calC z = z^*\M^{\T}H^{-1}z+z^*H^{-1}\M z = \bar\mu z^*H^{-1}z + \mu z^*H^{-1} z = 2\Ree(\mu)\,z^*H^{-1}z,
\]
 where the second equality follows from the fact that  $\M$ is real. Since $H \succ 0$, this gives $\Ree(\mu) \leq 0$.
If $\Ree(\mu)=0$, then $\calC z=0$, contradicting the fact we showed above. Hence every eigenvalue of $\M$ has
strictly negative real part, so $\M$ is Hurwitz.

\end{proof}


For $a=(a_1,a_2)^\T\in\R^2$ and $b=(b_1,b_2)^\T\in\R^2$, let us define the matrix
\begin{equation}\label{eq:def-G-a-b}
\G(a,b)=
\begin{pmatrix}
a_2b_1-a_1b_2 & b_1\\
b_1 & b_2
\end{pmatrix}.
\end{equation}
To motivate this definition, we note that
\[ \G(a,b)e = b,\qquad -J\G(a,b) + \G(a,b)J^\T = ba^\T - ab^\T.\]
Thus, $G_i = \G(u_i,v_i)$ is a natural solution of~\eqref{eq:completion} when $m = 1$.
We now develop a `gluing' lemma that allows one to combine these one-block solutions to obtain a matrix $X$ satisfying the linear equations in~\eqref{eq:completion}. When parameters $\lambda_i$ are sufficiently separated, $X$ remains close to the block diagonal matrix $G_0 = \diag(G_1,\ldots,G_m)$ while preserving the $(1,1)$ entry of each diagonal block exactly. 


\begin{lemma}[Gluing lemma]\label{lem:completion}
Let $m$ be a positive integer, $\lambda_1,\ldots,\lambda_m$ be distinct real scalars, and $u_i,v_i\in \R^2$ for $i\in[m]$. Let $D,u,v,c$ be defined as in~\eqref{eq:stacks} and $\G(\cdot,\cdot)$ be defined as in~\eqref{eq:def-G-a-b}. Let $G_i=\G(u_i,v_i)$ for $i\in [m]$, and
\(
G_0=\diag(G_1,\ldots,G_m) \in \R^{2m\times 2m}.
\)
Then, there is a unique symmetric matrix $X\in\R^{2m\times2m}$
satisfying~\eqref{eq:completion}. Furthermore, writing
$X=(X_{ij})_{i,j=1}^m$ in $2\times2$ blocks, we have
\begin{equation}\label{eq:preservation-property}
    (X_{ii})_{11}=(G_i)_{11}\quad(1\le i\le m).
\end{equation}
Let $\beta=\max_{1\le i\le m}\{\norm{u_i},\norm{v_i}\}$. In addition, if $h\ge4$ and 
\(
|\lambda_i-\lambda_j|\ge h\) for all \(i\ne j,
\)
then
\begin{equation}\label{eq:completion-bounds}
 \norm{X-G_0}\le \frac{12m\beta^2}{h}.
\end{equation}
\end{lemma}

\begin{proof}
We first construct a solution $X$ to~\eqref{eq:completion}. Assume $i,j\in [m]$ and $i\neq j$.
By expanding the second equation in~\eqref{eq:completion}, the $(i,j)$ block $X_{ij}$ in a solution $X$ satisfies
\begin{equation}\label{eq:sylvester}
 (\lambda_i-\lambda_j)X_{ij}-JX_{ij}+X_{ij}J^{\T}
       =v_i u_j^{\T}-u_i v_j^{\T}.
\end{equation}
\ac{The existence and uniqueness of the solution of equation~\eqref{eq:sylvester} follow from the distinctness of $\lambda_i$ and $\lambda_j$ by, e.g., \cite[p.231]{Bellman}.}
Taking the transpose of~\eqref{eq:sylvester} and exchanging $i$ and $j$ shows, by uniqueness, that $X_{ji} = X_{ij}^\T$.

It remains to construct the diagonal submatrices $X_{ii}$.
For $z=(z_1,z_2)^{\T}\in\R^2$, let
$\Gamma(z)=\left(\begin{smallmatrix}0&z_1\\z_1&z_2\end{smallmatrix}\right)$. 
Set
\begin{equation}\label{eq:diagonal-correction}
 t_i=\sum_{\substack{1\le j\le m\\j\ne i}}X_{ij}e\in\R^2.
\end{equation}
We show that $X_{ii}=G_i-\Gamma(t_i)$ will satisfy~\eqref{eq:completion} and~\eqref{eq:preservation-property}. For~\eqref{eq:preservation-property}, since $\Gamma(z)_{11}=0$, we have $(X_{ii})_{11}=(G_i)_{11}$. 
For~\eqref{eq:completion}, direct calculation gives
\begin{equation}\label{eq:local-identities}
 \begin{gathered}
 G_i e=v_i,\ -JG_i+G_iJ^{\T}=v_i u_i^{\T}-u_i v_i^{\T},\ 
 \Gamma(z)e=z,\ J\Gamma(z)=\Gamma(z)J^{\T}.
 \end{gathered}
\end{equation}
Thus~\eqref{eq:diagonal-correction} enforces $Xc=v$ and $D_{i}X_{ii} - X_{ii}D_{i}^\T = v_iu_i^\T - u_iv_i^\T$.

Next we show that the symmetric solution $X$ is unique.
Suppose $X$ and $\widetilde X$ are two symmetric solutions of
\eqref{eq:completion}, and let $Z=X-\widetilde X$. Then, we have
\[
Zc=0,\qquad DZ-ZD^\T=0.
\]
The off-diagonal blocks are zero by the uniqueness of the solution to
\eqref{eq:sylvester}. For $i\in [m]$, write the $(i,i)$ block of $Z$ as
\(
Z_{ii}=\begin{pmatrix}a&b\\ b&d\end{pmatrix}.
\)
Then $DZ-ZD^\T=0$ gives
\[
-JZ_{ii}+Z_{ii}J^\T
=
\begin{pmatrix}0&a\\-a&0\end{pmatrix}=0,
\]
so $a=0$. Since all off-diagonal blocks of $Z$ are $0$, $Zc=0$ gives $Z_{ii}e=0$, and hence $b=d=0$.
Thus, $Z_{ii}=0$ for every $i$, so $Z=0$.

We now show that the matrix $X$ constructed above satisfies the bound in~\eqref{eq:completion-bounds}.
Applying the triangle inequality to~\eqref{eq:sylvester} yields
\[
|\lambda_i-\lambda_j|\,\|X_{ij}\|
\le
\|JX_{ij}\|
+
\|X_{ij}J^{\T}\|
+
\|v_i u_j^{\T}\|
+
\|u_i v_j^{\T}\|.
\]
By definition of $\beta$ and the fact that $\|J\| = 1$, we have \((|\lambda_i-\lambda_j|-2)\norm{X_{ij}}\le2\beta^2.\)
Since $|\lambda_i-\lambda_j|\ge h$ and $h\ge4$, we have
$|\lambda_i-\lambda_j|-2\ge h/2$, and hence
\begin{equation}\label{eq:off-diagonal-bound}
 \norm{X_{ij}}\leq \frac{4\beta^2}{h}.
\end{equation}




We then bound $\|X-G_0\|$. First we have
\[
\|X_{ii}-G_i\|
=\|\Gamma(t_i)\|
\le 2\|t_i\|
\le 2\sum_{j\ne i}\|X_{ij}\|.
\]
Since $X-G_0$ is symmetric, for any block
vector $z=(z_1^\T,\ldots,z_m^\T)^\T$ we have
\[
\begin{aligned}
\bigl|z^T(X-G_0)z\bigr|
&\le
\sum_i \|X_{ii}-G_i\|\,\|z_i\|^2
+\sum_{i\ne j}\|X_{ij}\|\,\|z_i\|\,\|z_j\| \\
&\le
\sum_i
\left(
\|X_{ii}-G_i\|+\sum_{j\ne i}\|X_{ij}\|
\right)\|z_i\|^2,
\end{aligned}
\]
where the second inequality uses
$2\|z_i\|\|z_j\|\le \|z_i\|^2+\|z_j\|^2$ and
$\|X_{ij}\|=\|X_{ji}\|$. Therefore, using \eqref{eq:off-diagonal-bound}, we have
\[
\|X-G_0\|
\le
\max_i\left(
\|X_{ii}-G_i\|+\sum_{j\ne i}\|X_{ij}\|
\right) \le
3\max_i\sum_{j\ne i}\|X_{ij}\|
\le
\frac{12m\beta^2}{h}.
\]

\end{proof}

The proof of Theorem~\ref{thm:main} relies on a reduction from \SAT, which is well known to be (strongly) NP-hard~\cite{GareyJohnson}.
Recall that \SAT\ is a Boolean satisfiability problem in which a formula is expressed as a conjunction of clauses, where each clause consists of exactly three literals connected by logical OR. The problem asks whether there exists an assignment of truth values to the variables that makes every clause true. More specifically, a \SAT\ instance $\Phi$ with variables \(
x_1,\ldots,x_n
\)
and clauses
\(
C_1,\ldots,C_r\) is of the form
\[
\Phi = C_1 \wedge \ldots\wedge C_r,
\]
the $j$th clause reads \(
C_j=\ell_{j1}\vee \ell_{j2}\vee \ell_{j3}
\). Here, each $\ell_{jk}$ is called a literal and is either one of the variables
$x_1,\ldots,x_n$ or the negation of one of these variables. We identify a Boolean
assignment with a vector $x=(x_1,\ldots,x_n)\in\{\pm1\}^n$, where $x_i=1$
means that the $i$th variable is True and $x_i=-1$ means that it is False.
The \SAT\ instance $\Phi$ is satisfiable if there exists such an assignment
$x\in\{\pm1\}^n$ under which every clause $C_j$ evaluates to True.


For each clause index $j\in[r]$ and literal position
$k\in[3]$, define
\[
y_{jk}(x)
=
\begin{cases}
x_i, & \ell_{jk}=x_{i},\\
-x_i, & \ell_{jk}=\neg x_{i}.
\end{cases}
\]
Then,
\(
y_{jk}(x) = 1
\) when $\ell_{jk}$ is true under the
Boolean assignment $x$, and equals $-1$ when $\ell_{jk}$ is
false. For $j\in [r]$, we define an affine clause score function 
\begin{equation}\label{eq:clause-form}
 L_j(x)=3+2\sum_{k=1}^3y_{jk}(x).
\end{equation}
On Boolean inputs, \(L_j(x)=-3\) when clause \(C_j\) is false, while \(L_j(x)\in\{1,5,9\}\) when it is true.
For $i\in [n]$, we define
\begin{equation}\label{eq:def-u-v-g-1}
    u_i(x) = \begin{pmatrix}1\\3 x_i\end{pmatrix},\  v_i(x)=\begin{pmatrix}x_i\\1\end{pmatrix}, \ G_i(x)=\G(u_i(x),v_i(x)) = \begin{pmatrix}3x_i^2-1&x_i\\x_i&1\end{pmatrix}.
\end{equation}
For $j\in [r]$, we define 
\begin{equation}\label{eq:def-u-v-g-2}
u_{n+j}(x) =\begin{pmatrix}-1\\0\end{pmatrix},\ v_{n+j}(x)=\begin{pmatrix}0\\L_j(x) \end{pmatrix},\ G_{n+j}(x)=\G(u_{n+j}(x),v_{n+j}(x)) = L_j(x)I_2.    
\end{equation}

\begin{lemma}[Properties of the \SAT\ encoding]\label{lem:property-encoding}
    Suppose $\Phi$ is a \SAT\ instance with $n$ variables and $r$ clauses. For $i\in [n+r]$ and $x\in \R^n$, let $u_i(x), v_i(x)$ and $G_i(x)$ be defined as in~\eqref{eq:def-u-v-g-1} and~\eqref{eq:def-u-v-g-2}. Then, the following statements hold:
    \begin{enumerate}[label=\textup{(\roman*)},leftmargin=*,topsep=3pt,itemsep=2pt]
\item If $x$ is a satisfying Boolean assignment to $\Phi$, then $$G_i(x)\succeq \frac{1}{3}I_2 \quad(1\leq i\leq n+r) \ \text{  and  } \ \max_{1\le i\le n+r}\{\norm{u_i(x)},\norm{v_i(x)}\}\le9.$$ 
\item If there exists $x\in \R^n$ such that $(G_k(x))_{11}\geq 0$ for all $k\in [n + r]$, then $\Phi$ is satisfiable.
\end{enumerate}
\end{lemma}
\begin{proof}
    If $x$ is a satisfying Boolean assignment for $\Phi$, then for $i \in [n]$, $x_i = \pm 1$. Therefore, $G_i(x) = \begin{pmatrix}2&1\\1&1\end{pmatrix}$ or $ \begin{pmatrix}2&-1\\-1&1\end{pmatrix}$, both satisfy $G_i(x)\succeq \frac{1}{3}I_2$. We have $\norm{u_i(x)} = \sqrt{10}$ and $\norm{v_i(x)} = \sqrt{2}$.
    Since $x$ is a satisfying Boolean assignment for $\Phi$, for all $j\in[r]$ we have $L_j(x) \in \{1,5,9\}$. Thus, $G_{n+j}(x)\succeq \frac{1}{3}I_2$,  $\norm{u_{n+j}(x)} = 1$ and $\norm{v_{n+j}(x)} \leq 9$. This proves claim (i).

    If there exists $x\in\R^n$ such that $(G_k(x))_{11}\geq 0$ for all $k\in [n + r]$, then  $3x_i^2 - 1\geq 0$ for all $i\in[n]$ and $L_j(x)\ge0$ for all $j\in [r]$. For every index $i\in [n]$, we construct $\hat{x}_i = \operatorname{sign}(x_i)$. We prove that $\hat{x}$ is a satisfying assignment to $\Phi$. Suppose for the sake of contradiction that there exists $j\in[r]$ such that $L_j(\hat{x}) < 0$. Then we know that
    \[y_{j1}(\hat{x})=-1,\ y_{j2}(\hat{x})=-1,\ y_{j3}(\hat{x})=-1.\]
    Since $|x_i|\geq \frac{1}{\sqrt{3}} > \frac{1}{2}$, we have
    \[y_{j1}(x)< -\frac{1}{2},\ y_{j2}(x)< -\frac{1}{2},\ y_{j3}(x)< -\frac{1}{2}.\]
    However, this gives
    \(L_j(x) < 0\), a contradiction. Thus, $L_j(\hat{x}) \geq 0$ for all $j\in [r]$, which indicates that $C_j$ is true under assignment $\hat x$ for all $j\in [r]$.
    This proves claim (ii).
    \end{proof}

We can now present the proof of Theorem~\ref{thm:main}.
\begin{proof}[Proof of Theorem~\ref{thm:main}]

\smallskip
\noindent\emph{Step 1: Simultaneous stabilization.}

Given a \SAT\ instance $\Phi$ with $n$ variables and $r$ clauses, let $m = n +r$. Define $h=4096m$ and for $i \in [m]$ define $\lambda_i = ih$ and $D_i = \lambda_i I_2 - J$. Let $u_i(x), v_i(x)$, and $G_i(x)$ be defined as in~\eqref{eq:def-u-v-g-1} and~\eqref{eq:def-u-v-g-2}. For notational convenience, for $i\in [n+r]$, we write $u_i(x)=u_i^0+U_ix$ and $v_i(x)=v_i^0+V_ix$, where
$u_i^0,v_i^0\in\Z^2$ and $U_i,V_i\in\Z^{2\times n}$.
Construct the following matrices $A_i\in\Z^{4\times4}$,
$B_i\in\Z^{4\times n}$, and $C_i\in\Z^{1\times4}$:
\begin{equation}\label{eq:local-plants}
 A_i=\begin{pmatrix}0&-D_i-u_i^0e^{\T}\\I_2&-v_i^0e^{\T}\end{pmatrix},
 \qquad B_i=-\binom{U_i}{V_i},\qquad C_i=(0,0,0,1).
\end{equation}
We show that $\Phi$ is satisfiable if and only if there exists $x\in\R^n$ such that $A_i + B_ixC_i$ is Hurwitz for all $i\in [m]$.

If $\Phi$ is satisfiable, let $x$ be a satisfying Boolean assignment.
We claim that for every $i \in [m]$, $A_i+B_ixC_i$ is Hurwitz. A straightforward calculation gives
\[
\begin{aligned}
A_i+B_i x C_i
&=
\begin{pmatrix}
0 & -D_i-u_i(x)e^\T\\
I_2 & -v_i(x)e^\T
\end{pmatrix}.
\end{aligned}
\]
 One can verify that for all $i\in [m]$, we have $D_iG_i(x) - G_i(x)D_i^\T = v_i(x)u_i(x)^\T -  u_i(x)v_i(x)^\T$ and $G_i(x)e = v_i(x)$. By Lemma~\ref{lem:property-encoding}(i), $G_i(x) \succ 0$ and $\|u_i(x)\| \leq 9$ for all $i$. Therefore,
\[
 \sym(D_i+u_i(x) e^{\T})
 \succeq(\lambda_i-1-\norm{u_i(x)})I_2
 \succeq(h-10)I_2\succ0.
\]
Then, applying Lemma~\ref{lem:certificate}(ii) with $\lambda_i, u_i(x), v_i(x),D_i$ shows that $A_i+B_ixC_i$ is Hurwitz for all $i\in [m]$.

Conversely, suppose there exists $x\in\R^n$ such that  $A_i+B_ixC_i$ is Hurwitz for all $i\in [m]$. Lemma~\ref{lem:certificate}(i) shows that $G_i(x)\succeq0$ for all $i\in[m]$. Thus, $(G_i(x))_{11} \geq 0$ for all $i\in [m]$. Applying Lemma~\ref{lem:property-encoding}(ii) gives that $\Phi$ is satisfiable.

\smallskip
\noindent
\emph{Step 2: Output feedback stabilization.}

We now combine the matrices constructed in Step 1 in a specific way to conclude the proof. We define 
\[
u^0:=
\begin{pmatrix}
u_1^0\\
\vdots\\
u_m^0
\end{pmatrix}\in\Z^{2m},
\
v^0:=
\begin{pmatrix}
v_1^0\\
\vdots\\
v_m^0
\end{pmatrix}\in\Z^{2m}, \ U:=
\begin{pmatrix}
U_1\\
\vdots\\
U_m
\end{pmatrix} \in\Z^{2m\times n},
\
V:=
\begin{pmatrix}
V_1\\
\vdots\\
V_m
\end{pmatrix}\in\Z^{2m\times n}.
\]
For every $x\in \R^n$, we write
$u(x)=u^0+Ux$ and $v(x)=v^0+Vx$.
Let $D,c$ be defined as in~\eqref{eq:stacks} and $N=4m$. Let
\[
A=\begin{pmatrix}0&-D-u^0c^{\T}\\I_{2m}&-v^0c^{\T}\end{pmatrix} \in \Z^{N\times N},
 \qquad B=-\binom UV\in \Z^{N\times n},\qquad C=(0_{1\times2m},c^{\T})\in\Z^{1\times N}.
\]
We claim that $\Phi$ is satisfiable if and only if there exists $x\in\R^n$ such that $A + BxC$ is Hurwitz. By definition of $A,B$, and $C$, we have
\[
A+BxC=\begin{pmatrix}0&-D-u(x)c^{\T}\\I_{2m}&-v(x)c^{\T}\end{pmatrix}.
\]
By Lemma~\ref{lem:completion}, for every $x\in\R^n$ there is a symmetric matrix $X(x)\in\R^{2m\times 2m}$ such that
\[
X(x)c = v(x),\ \  DX(x) - X(x)D^\T = v(x)u(x)^\T - u(x)v(x)^\T.
\] 
Writing $X(x)=(X_{ij}(x))_{i,j=1}^m$ in
$2\times2$ blocks, we have
\((X_{ii}(x))_{11}=(G_i(x))_{11}\) for all $i\in [m]$.


Suppose $\Phi$ has a satisfying Boolean assignment $x\in \{\pm1\}^n$. By Lemma~\ref{lem:property-encoding}(i), $$\max_{1\le i\le m}\{\norm{u_i({x})},\norm{v_i({x})}\}\le9.$$ Together with Lemma~\ref{lem:completion}, we have
\begin{equation}\label{eq:step-2-diagonal-bound}
    \norm{X({x})-\diag(G_1({x}),\ldots,G_m({x}))}
 \le\frac{12m\cdot9^2}{4096m}
 =\frac{243}{1024}<\frac14.
\end{equation}
Also from Lemma~\ref{lem:property-encoding}(i), we know that $G_i({x})\succeq \frac{1}{3}I_2$ for all $i\in [m]$. Combining with~\eqref{eq:step-2-diagonal-bound} gives
\( X({x})\succ\tfrac1{12}I_{2m}\).
Since $\norm{u({x})}\le9\sqrt m$ and $\norm{c}=\sqrt m$, we have
\[
 \sym(D+u({x})c^{\T})
 \succeq(h-1-\norm{u({x})}\norm{c})I_{2m}
 \succeq(h-1-9m)I_{2m}\succ0.\]
Then applying Lemma~\ref{lem:certificate}(ii) gives that $A+B{x}C$ is Hurwitz.

Conversely, suppose there is $x\in\R^n$ such that $A + BxC$ is Hurwitz. By Lemma~\ref{lem:certificate} (i), $X(x)$ is positive semidefinite. By Lemma~\ref{lem:completion}, for $i\in [m]$, \(
 (G_i(x))_{11}=(X_{ii}(x))_{11}\ge0.
\) Then applying Lemma~\ref{lem:property-encoding}(ii) shows that $\Phi$ is satisfiable. 

The reduction uses only polynomially bounded integers.  We have
$N=4m$, $h=4096m$, and
\[
\lambda_i=ih\le \lambda_m=4096m^2=256N^2.
\]
It is not difficult to verify that every entry of $A$ has absolute value at most $256N^2$, every entry of $B$ has absolute value at most $6$, and $C$ is binary.  Therefore, the reduction can be carried out in polynomial time and establishes strong NP-hardness.

\end{proof}

\section{Discrete-time hardness}\label{sec:dt}

In this section, we establish the discrete-time analogue of Theorem \ref{thm:main}:

\begin{theorem}\label{thm:dt-hardness}
Given matrices
$A_d\in\Q^{N\times N}$, $B_d\in\Q^{N\times p}$, and $C_d\in\Q^{1\times N}$, deciding if there exists a matrix $K_d\in\R^{p\times1}$ that makes $A_d+B_dK_dC_d$ Schur is
strongly NP-hard.
\end{theorem}

The key idea of the proof (modulo an important technical problem that we shall address) is to utilize the Cayley transform to transfer the reduction from the continuous-time case. The Cayley transform is a well-studied tool for passing between continuous and discrete time systems (see, for example, \cite{Staffans}). 

\begin{lemma}[Cayley transform]\label{lem:cayley}
Let $A\in\Q^{N\times N}$, $B\in\Q^{N\times p}$, and
$C\in\Q^{1\times N}$. Choose $\beta\in\Q$, $\beta>0$, such that
$\beta I-A$ is invertible, and denote
\begin{equation}\label{eq:cayley-data}
R_\beta=(\beta I-A)^{-1},\qquad \widehat{D}=CR_\beta B, \qquad
A_d=2\beta R_\beta-I,\qquad B_d=R_\beta B,\qquad C_d=2\beta CR_\beta.
\end{equation}
Let $\mathcal{K}_c$ denote the set of all stabilizing controllers of the system $(A,B,C)$ in continuous time and $\mathcal{K}_d$ the set of all stabilizing controllers of the system $(A_d, B_d, C_d)$ in discrete time. Then, the maps
\begin{align}\label{eq:cayley-gains}
\Psi_{c\to d}:\mathcal K_c&\longrightarrow\mathcal K_d, 
\qquad
\Psi_{c\to d}(K_c)=\frac{K_c}{1-\widehat{D}K_c},\\
\Psi_{d\to c}:\mathcal K_d&\longrightarrow\mathcal K_c,
\qquad
\Psi_{d\to c}(K_d)=\frac{K_d}{1+\widehat{D}K_d},
\end{align}
are inverses and provide a bijection between the rational stabilizing
controllers.
\end{lemma}

\begin{proof}
Let $F_c=A+BK_cC$ and $F_d=A_d+B_dK_dC_d$. Since
$R_\beta=(\beta I-A)^{-1}$, the matrix determinant lemma gives that
\begin{align*}
\det(\beta I-F_c)
&=\det(\beta I-A-BK_cC)
 =\det(\beta I-A)\bigl(1-CR_\beta BK_c\bigr) \\
 &=\det(\beta I-A)(1-\widehat{D}K_c),\\
\det(I+F_d)
&=\det\!\bigl(2\beta R_\beta+2\beta R_\beta BK_dCR_\beta\bigr)
 =(2\beta)^N\det(R_\beta)\det(I+BK_dCR_\beta) \\
 &=(2\beta)^N\det(R_\beta)(1+\widehat{D}K_d).
\label{eq:dt-exception}
\end{align*}

If $K_c$ is stabilizing, then $F_c$ must be Hurwitz and thus $\beta$ is not an eigenvalue of $F_c$. Therefore, $1-\widehat{D}K_c\ne0$. Similarly, if $K_d$ is stabilizing, then $F_d$ must be Schur, and thus $-1$ is not an eigenvalue of $F_d$, implying that $1+\widehat{D}K_d\ne0$. It follows that $\Psi_{c\to d}$ is well-defined for every
$K_c\in\mathcal K_c$, and that $\Psi_{d\to c}$ is well-defined
for every $K_d\in\mathcal K_d$. We next show that these mappings
preserve stability. 

We begin by applying the Sherman--Morrison formula to
$\beta I-F_c=(\beta I-A)-BK_cC$. This gives us
\begin{equation}\label{eq:rank-one-cayley}
(\beta I-F_c)^{-1}
=R_\beta+\frac{R_\beta BK_cCR_\beta}{1-\widehat{D}K_c}.
\end{equation}
Using \eqref{eq:cayley-data}, \eqref{eq:cayley-gains}, and \eqref{eq:rank-one-cayley}, we have
\begin{align}
F_d
&=2\beta R_\beta-I+2\beta R_\beta BK_dCR_\beta\notag\\
&=2\beta\left(R_\beta+\frac{R_\beta BK_cCR_\beta}{1-\widehat{D}K_c}\right)-I\notag\\
&=2\beta(\beta I-F_c)^{-1}-I
=(\beta I-F_c)^{-1}(\beta I+F_c).
\label{eq:closed-loop-cayley}
\end{align}
Now, let $s$ be an eigenvalue of $F_c$. Equation~\eqref{eq:closed-loop-cayley}
shows that the corresponding eigenvalue of $F_d$ is $z=\frac{\beta+s}{\beta-s}.$
Conversely, every eigenvalue of $F_d$ is obtained from an eigenvalue of
$F_c$ through this transformation. Since
\[
\abs{\beta+s}^2-\abs{\beta-s}^2
=4\beta\Ree s,
\]
and $\beta>0$, we have
\[
\abs z<1
\quad\Longleftrightarrow\quad
\abs{\beta+s}<\abs{\beta-s}
\quad\Longleftrightarrow\quad
\Ree s<0.
\]
Therefore, $F_c$ is Hurwitz if and only if $F_d$ is Schur. The fact that $\Psi_{c \to d}$ and $\Psi_{d \to c}$ are mutual inverses follows from a direct computation:
\begin{align*}
\Psi_{d\to c}\!\left(\Psi_{c\to d}(K_c)\right)
&=
\frac{\dfrac{K_c}{1-\widehat{D}K_c}}
     {1+\widehat{D}\dfrac{K_c}{1-\widehat{D}K_c}}
=
\frac{K_c}{1-\widehat{D}K_c+\widehat{D}K_c}
=
K_c,\\
\Psi_{c\to d}\!\left(\Psi_{d\to c}(K_d)\right)
&=
\frac{\dfrac{K_d}{1+\widehat{D}K_d}}
     {1-\widehat{D}\dfrac{K_d}{1+\widehat{D}K_d}}
=
\frac{K_d}{1+\widehat{D}K_d-\widehat{D}K_d}
=
K_d.
\end{align*}
Finally, since $\beta I-A$ is an invertible rational matrix,
$R_\beta=(\beta I-A)^{-1}$ and $\widehat{D}=CR_\beta B$ are rational. Hence
\eqref{eq:cayley-gains} shows that both transformations
preserve rationality.
\end{proof}

Since the inverse of a rational matrix can be computed in polynomial time, a direct application of Lemma~\ref{lem:cayley} gives
\(\mathrm{NP}\)-hardness of output feedback stabilization in discrete time. For strong
\(\mathrm{NP}\)-hardness, however, this is not enough: the numerators and denominators of the transformed data must themselves
have polynomial magnitude and this can in general fail under matrix inversion. 

To control these denominators, we modify the parameters $\lambda_i$ used in the continuous-time construction. With the original choice $\lambda_i=ih$, the blocks $(I_2+D_i)^{-1}$ involve denominators $(1+ih)^2$ that vary with $i$. Although each denominator has polynomial magnitude, this alone does not control the denominators arising when the blocks are combined through the Sherman--Morrison formula. We therefore choose parameters for which these blocks share a common denominator of polynomial magnitude.

\begin{proof}[Proof of Theorem \ref{thm:dt-hardness}]
Let $\Phi$ be a $\SAT$ formula with $n\ge1$ variables and $r$
clauses, and set $m=n+r$ and $N=4m$. Let $u_i(x)$ and $v_i(x)$
be the vectors defined in~\eqref{eq:def-u-v-g-1}
and~\eqref{eq:def-u-v-g-2}. Collect these vectors into
\[
u(x)=
\begin{pmatrix}u_1(x)\\ \vdots\\ u_m(x)\end{pmatrix}
=u^0+Ux,
\qquad
v(x)=
\begin{pmatrix}v_1(x)\\ \vdots\\ v_m(x)\end{pmatrix}
=v^0+Vx.
\]
Here $u^0,v^0\in\Z^{2m}$ contain the constant terms, and
$U,V\in\Z^{2m\times n}$ contain the coefficients of $x$.
By~\eqref{eq:def-u-v-g-1} and~\eqref{eq:def-u-v-g-2},
every entry of $u^0,v^0$ has absolute value at most $3$,
and every entry of $U,V$ has absolute value at most $6$.

We modify the choice of parameters by setting
\begin{equation}\label{eq:dt-separation}
h=4096m,\qquad
\tau=hm(m+1),\qquad
\lambda_i=\frac{\tau}{i}-1
\quad(1\le i\le m).
\end{equation}
As in \eqref{eq:stacks}, define
\[
D_i=\lambda_iI_2-J,\qquad
D=\diag(D_1,\ldots,D_m),\qquad
c=(e^\T,\ldots,e^\T)^\T,
\]
and
\begin{equation}\label{eq:dt-continuous-instance}
A_0=
\begin{pmatrix}
0&-D\\
I_{2m}&0
\end{pmatrix},
\qquad
b_0=\binom{u^0}{v^0},
\qquad
B=-\binom{U}{V},
\qquad
C=(0_{1\times2m},c^\T),
\qquad
A=A_0-b_0C.
\end{equation}
For every $x\in\R^n$, the matrix $A+BxC$ has the form $\M$
in~\eqref{eq:realization}, with $u=u(x)$ and $v=v(x)$. Note that the parameters in~\eqref{eq:dt-separation} satisfy
\[
\lambda_i\ge\frac{\tau}{m}-1=h(m+1)-1\ge h
\qquad(i\in[m]),
\]
and, for $i\ne j$,
\[
|\lambda_i-\lambda_j|
=\tau\frac{|i-j|}{ij}
\ge\frac{\tau}{m^2}
=h\frac{m+1}{m}
>h.
\]
One can verify that the argument in Step~2 of the proof of
Theorem~\ref{thm:main}, using Lemmas~\ref{lem:certificate},
\ref{lem:completion}, and~\ref{lem:property-encoding}, applies
whenever $\lambda_i\ge h$ and $|\lambda_i-\lambda_j|\ge h$
for $i\ne j$, with $h=4096m$.
Since the parameters in~\eqref{eq:dt-separation} satisfy
these conditions, we obtain
\[
\Phi\text{ is satisfiable}
\quad\Longleftrightarrow\quad
\exists x\in\R^n:\ A+BxC\text{ is Hurwitz}.
\]

\emph{Passing to discrete time.}
We now apply Lemma~\ref{lem:cayley} with $\beta=1$.
First, we verify that $I_N-A$ is invertible. To invert
$I_N-A_0$, define
\[
S_i:=(I_2+D_i)^{-1}\quad(1\le i\le m),
\qquad
S:=\diag(S_1,\ldots,S_m)=(I_{2m}+D)^{-1}.
\]
Since \(1+\lambda_i=\tau/i\) and \(J^2=0\), we have that $S_i
=\left(\frac{\tau}{i}I_2-J\right)^{-1}
=\frac{i}{\tau}I_2+\frac{i^2}{\tau^2}J.$ Moreover,
\[
I_N-A_0=
\begin{pmatrix}
I_{2m}&D\\
-I_{2m}&I_{2m}
\end{pmatrix},
\]
so block inversion gives
\[
R_0:=(I_N-A_0)^{-1}
=
\begin{pmatrix}
S&S-I_{2m}\\
S&S
\end{pmatrix}.
\]
The blocks $u_i^0,v_i^0$ of the constant vectors $u^0,v^0$,
obtained from~\eqref{eq:def-u-v-g-1}
and~\eqref{eq:def-u-v-g-2}, satisfy
\[
u_i^0=\binom{\epsilon_i}{0},\qquad
v_i^0=\binom{0}{\eta_i},\qquad
(\epsilon_i,\eta_i)=
\begin{cases}
(1,1),&1\le i\le n,\\
(-1,3),&n<i\le m.
\end{cases}
\]
The expressions for $R_0$, $C$, and $b_0$ give
\[
CR_0=(c^\T S,c^\T S),\qquad
\tau^2e^\T S_i=(i^2,i\tau).
\]
Consequently,
\[
1+CR_0b_0
=1+\frac{1}{\tau^2}
\sum_{i=1}^m
\bigl(\tau\eta_i i+\epsilon_i i^2\bigr).
\]
For convenience, define
\[
T:=\tau^2R_0,\qquad
q:=\tau^2+\sum_{i=1}^m\bigl(\tau\eta_i i+\epsilon_i i^2\bigr).
\]
Because \(\eta_i\ge1\), \(\epsilon_i\ge-1\), and \(\tau>m\ge i\), we have that $\tau\eta_i i+\epsilon_i i^2
\ge i(\tau-i)>0$ and therefore
\[
1+CR_0b_0=\frac{q}{\tau^2}>0.
\]
Thus the Sherman--Morrison denominator is nonzero. Since
$I_N-A=(I_N-A_0)+b_0C$, the matrix $I_N-A$ is invertible and the
Sherman--Morrison formula gives
\[
R:=(I_N-A)^{-1}
=R_0-\frac{R_0b_0CR_0}{1+CR_0b_0}.
\]
Set
\[
P:=qT-Tb_0CT,\qquad \Delta:=\tau^2q.
\]
Then $R=P/\Delta$. The transformed matrices in
Lemma~\ref{lem:cayley}, with $\beta=1$, are
\[
A_d=2R-I_N=\frac{2P-\Delta I_N}{\Delta},\qquad
B_d=RB=\frac{PB}{\Delta},\qquad
C_d=2CR=\frac{2CP}{\Delta}.
\]
Identifying the continuous-time controller \(K_c\) with \(x\),
Lemma~\ref{lem:cayley} now gives
\[
\begin{aligned}
\Phi\text{ is satisfiable}
&\quad\Longleftrightarrow\quad
\exists K_c\in\R^{n\times1}:
A+BK_cC\text{ is Hurwitz}\\
&\quad\Longleftrightarrow\quad
\exists K_d\in\R^{n\times1}:
A_d+B_dK_dC_d\text{ is Schur}.
\end{aligned}
\]
It remains to show that these matrices can be computed in polynomial
time and that the numerators and denominators of their entries have
polynomial magnitude. We establish these properties using the
explicit expressions above.

We begin by observing that since \(\tau\) and \(J\) are integral, our expression for \(S_i\) gives that $\tau^2S_i=i\tau I_2+i^2J\in\Z^{2\times2}.$
It then follows from the block expression for \(R_0\) that $T\in\Z^{N\times N}.$
Moreover, since \(\tau\), \(\epsilon_i\), and \(\eta_i\) are integers,
$q$ is a positive integer. Thus $P\in\Z^{N\times N}$ and $\Delta\in\Z_{>0}$.
Now, since $\tau=O(m^3),$ and $N = 4m$, we have \(\tau=N^{O(1)}\). The expressions for \(S_i\) and \(R_0\),
together with \(\tau>m\ge i\), give
\[
\abs{T_{ab}}\le\tau^2
\qquad(1\le a,b\le N).
\]
In addition, since every entry of $b_0$ has absolute value at most $3$ and \(C\) is binary, we have the bounds $\abs{(Tb_0)_a}
\le3N\tau^2$, $\abs{(CT)_b}
\le N\tau^2$, and 
\[
0<q\le(1+3N^2)\tau^2.
\]
Consequently, the absolute value of each entry of the integer
matrix $P=qT-Tb_0CT$ is at most $(1+6N^2)\tau^4 = N^{O(1)}$ while the common positive denominator $\Delta$ is at most $(1+3N^2)\tau^4 = N^{O(1)}.$

Finally, every entry of \(PB\) and \(CP\) is a sum of at most \(N\)
products. The entries of \(B\) have absolute value at most \(6\), while
\(C\) is binary. Therefore the integer numerators $2P-\Delta I_N$, $PB$, and $2CP$ of $A_d$, $B_d$, and $C_d$, respectively,
and their common positive denominator $\Delta$ all have magnitude \(N^{O(1)}\).
The displayed expressions can be evaluated using polynomially many
arithmetic operations on integers of polynomially bounded magnitude.
Reducing the resulting fractions to lowest terms also takes
polynomial time and cannot increase these bounds. Hence the reduction
is computable in polynomial time and establishes strong
$\mathrm{NP}$-hardness.
\end{proof}

\section{Representation obstruction to polynomial-time algorithms}
\label{sec:controller-size}

Independently of assumptions such as P$\ne$NP, we show in this section that there is a representation
obstruction to polynomial-time controller synthesis: on some stabilizable
instances, a polynomial-time algorithm cannot even write down a stabilizing
controller (i.e. feedback gain matrix) in its standard representation. This also suggests that output feedback
stabilization may not belong to NP, unless there is a polynomial-size certificate of stabilizability that avoids writing down the controller. The problem belongs to $\mathrm{PSPACE}$, however, since we have already argued that it can be reduced to testing the feasibility of polynomial inequalities.

Observe that if a stabilizing controller exists, a rational one exists as well.
Indeed, the sets of Hurwitz and Schur matrices are open, and
$K\mapsto A+BKC$ is continuous. The set of stabilizing controllers is
therefore open, and every nonempty open subset of a finite-dimensional
real vector space contains a rational point.

We encode an entry $a/b$ of $A,B,$ or $C$, with $a\in\mathbb Z$, $b\in\mathbb
Z_{>0}$, and $\gcd(|a|,b)=1$, by writing its signed numerator and positive
denominator explicitly in binary.
Dense matrix encoding lists every entry,
including zeros, which have constant encoding size. We use the elementary
observation that for $a,b\in\mathbb Z_{>0}$ and $\varepsilon>0$,
\begin{equation}\label{eq:bits-elementary}
 0<\frac ab\le\varepsilon
 \quad\Longrightarrow\quad
 b\ge\varepsilon^{-1}.
\end{equation}
This follows from $a\ge1$ and applies after cancellation, so reducing a
fraction cannot invalidate this denominator bound.

\begin{theorem}[Continuous time]\label{thm:bits-CT}
For each integer $n\ge1$, matrices
$A\in\mathbb Z^{N\times N}$, $B\in\mathbb Z^{N\times p}$, and a nonzero
row vector $C\in\{0,1\}^{1\times N}$ can be constructed in time polynomial in $n$,
with
\[
 N=8n+12,\qquad p=n+2,\qquad
 \max_{i,j}|A_{ij}|\le256N^2,\qquad \max_{i,j}|B_{ij}|\le6,
\]
such that a rational stabilizing controller exists, but every (real)
controller $K=(x_0,\ldots,x_n,t)^{\T}$ for which $A+BKC$ is Hurwitz satisfies
\begin{equation}\label{eq:bits-small-entry}
 0<x_n\le\frac23\,2^{-2^n}.
\end{equation}
In particular, writing $x_n=a/b$ in lowest terms with $b>0$ for any rational
stabilizing controller gives $b\ge\frac32\,2^{2^n}$; this denominator alone
requires more than $2^n$ bits, while the input size is $L_n=\Theta(n^2)$.
\end{theorem}

\begin{proof}
First, we describe the construction.
We will use the definition of $J,e,$ and $\G$ from Section~\ref{sec:ct}.
Set
$m=2n+3$, $h=4096m$, and, for $i=1,\dots,m$, set
$\lambda_i=ih$ and $D_i=\lambda_iI_2-J$. Let
\[
 D=\diag(D_1,\ldots,D_m)\in\mathbb R^{2m\times2m},\qquad
 c=\begin{pmatrix}e\\\vdots\\e\end{pmatrix}\in\mathbb R^{2m}.
\]
For $K=(x_0,\ldots,x_n,t)^{\T}\in\mathbb R^{p}$, define
$u_i(K),v_i(K)\in\mathbb R^2$ by
\begin{equation}\label{eq:bits-vectors}
\begin{aligned}
 u_{j+1}&=\binom{-x_j}{0},&v_{j+1}&=\binom{0}{t}
       &&(0\le j\le n),\\
 u_{n+2+j}&=\binom{x_{j+1}}{3x_j},&v_{n+2+j}&=\binom{x_j}{t}
       &&(0\le j<n),\\
 u_{m-1}&=\binom{6x_0-t}{0},&v_{m-1}&=\binom{0}{t},\\
 u_m&=\binom{t-2}{0},&v_m&=\binom{0}{t}.
\end{aligned}
\end{equation}
Stack these vectors vertically as $u=u^0+UK$ and $v=VK$, with
$u^0\in\mathbb Z^{2m}$ and $U,V\in\mathbb Z^{2m\times p}$. The only nonzero
entry of $u^0$ is its $(2m-1)$st entry, equal to $-2$. Define
\begin{equation}\label{eq:bits-plant}
 A=\begin{pmatrix}0&-D-u^0c^{\T}\\I_{2m}&0\end{pmatrix},\qquad
 B=-\binom UV,\qquad C=(0_{1\times2m},c^{\T}).
\end{equation}
Thus $M=A+BKC$ is the matrix $\M$ of Lemma~\ref{lem:certificate}.
For each controller, let $G_i=\G(u_i,v_i)$ and
$G_0=\diag(G_1,\ldots,G_m)$. Lemma~\ref{lem:completion} gives a symmetric
$X\in\mathbb R^{2m\times2m}$ satisfying
\begin{equation}\label{eq:bits-completion}
 Xc=v,\qquad DX-XD^{\T}=vu^{\T}-uv^{\T},\qquad
 (X_{ii})_{11}=(G_i)_{11}\quad(1\le i\le m).
\end{equation}

Second, we show that any stabilizing controller must satisfy \eqref{eq:bits-small-entry}.
Suppose $M$ is Hurwitz. Since $\operatorname{tr}M=-mt$, we have $t>0$, and by 
Lemma~\ref{lem:certificate}\textup{(i)}, $X\succeq0$. 
Thus $X$ has nonnegative diagonal entries.
From~\eqref{eq:bits-completion}, we then have that $(G_i)_{11}=(X_{ii})_{11}\geq 0$ for $i=1,\dots,m$.
Substituting the vectors in
\eqref{eq:bits-vectors} gives
\[
 tx_j\ge0\quad(0\le j\le n),\qquad
 3x_j^2-tx_{j+1}\ge0\quad(0\le j<n),
 \qquad t(t-6x_0)\ge0,\qquad t(2-t)\ge0.
\]
Since $t>0$, these inequalities yield:
\begin{equation}\label{eq:bits-chain}
 \begin{gathered}
 x_j\ge0\quad(0\le j\le n),\qquad x_0\le t/6,\qquad 0<t\le2,\\
 tx_{j+1}\le3x_j^2\quad(0\le j<n).
 \end{gathered}
\end{equation}

Moreover, we claim that $x_n>0$. Otherwise, the first coordinates of $u_{n+1}$ and
$v_{n+1}$ vanish. In the state partition
$(\xi,\eta)\in\mathbb R^{2m}\times\mathbb R^{2m}$, the coordinates with
index $q=2n+1$ then form the autonomous subsystem
\[
 \dot\xi_q=-\lambda_{n+1}\eta_q,\qquad \dot\eta_q=\xi_q.
\]
Its positive-definite energy $\xi_q^2+\lambda_{n+1}\eta_q^2$ is conserved,
contradicting convergence of every trajectory to zero.

Set $z_j=3x_j/t$ for $j=0,\dots,n$. Equations~\eqref{eq:bits-chain} give
$0\le z_0\le1/2$ and $0\le z_{j+1}\le z_j^2$. Induction therefore gives
\begin{equation}\label{eq:bits-normalized-chain}
 0\le\frac{x_j}{t}\le\frac13\,2^{-2^j}\le\frac16
 \quad(0\le j\le n).
\end{equation}
Together with $x_n>0$ and $t\le2$, this proves~\eqref{eq:bits-small-entry}.
The denominator bound follows from~\eqref{eq:bits-elementary}.

Third, we show the existence of a rational stabilizing controller.
Choose
\begin{equation}\label{eq:bits-witness}
 t=1,\qquad x_j=12^{-2^j}\quad(0\le j\le n).
\end{equation}
To prove that $X\succ0$, set $S_i=\diag(d_i,1)$ and
$S=\diag(S_1,\ldots,S_m)$, where
\begin{equation}\label{eq:bits-scaling}
 \begin{gathered}
 d_{j+1}=\sqrt{x_j}\quad(0\le j\le n),\qquad
 d_{n+2+j}=x_j\quad(0\le j<n),\\
 d_{m-1}=1/\sqrt2,\qquad d_m=1.
 \end{gathered}
\end{equation}
Every $d_i$ lies in $(0,1]$. Direct substitution shows that the matrices
$\widehat G_i=S_i^{-1}G_iS_i^{-1}$ are either $I_2$ or
$F=\left(\begin{smallmatrix}2&1\\1&1\end{smallmatrix}\right)$.
Since $F \succeq \frac{1}{3}I_2$ and by the definitions of $S_i,u_i,v_i$, we have
\begin{equation}\label{eq:bits-scaled-bounds}
 \widehat G_0:=S^{-1}G_0S^{-1}\succeq\tfrac13I_{2m},\qquad
 \widehat\beta:=\max_{1\le i\le m}
 \{\norm{S_i^{-1}u_i},\norm{S_i^{-1}v_i}\}\le2.
\end{equation}

Write $\widehat X=S^{-1}XS^{-1}$,
$\widehat u_i=S_i^{-1}u_i$, and $\widehat v_i=S_i^{-1}v_i$.
Since $S_i^{-1}JS_i=d_iJ$, the off-diagonal block equations are
\[
 (\lambda_i-\lambda_j)\widehat X_{ij}
 -d_iJ\widehat X_{ij}+d_j\widehat X_{ij}J^{\T}
 =\widehat v_i\widehat u_j^{\T}-\widehat u_i\widehat v_j^{\T}
 \quad(i\ne j).
\]
As $d_i,d_j\le1$ and $|\lambda_i-\lambda_j|\ge h\ge4$, the triangle inequality
gives $\norm{\widehat X_{ij}}\le4\widehat\beta^2/h$.
The diagonal correction also retains its form:
\[
 \widehat X_{ii}-\widehat G_i
 =-\Gamma\!\left(\sum_{j\ne i}\widehat X_{ij}e\right),\qquad
 \Gamma(z)=\begin{pmatrix}0&z_1\\z_1&z_2\end{pmatrix}.
\]
Indeed, $S_ie=e$ and
$S_i^{-1}\Gamma(S_i z)S_i^{-1}=\Gamma(z)$.
Using $\norm{\Gamma(z)}\le2\norm z$ and the block-row-sum bound as in
Lemma~\ref{lem:completion}, we obtain
\begin{equation}\label{eq:bits-weighted}
 \norm{\widehat X-\widehat G_0}
 \le3\max_{1\le i\le m}\sum_{j\ne i}\norm{\widehat X_{ij}}
 \le\frac{12m\widehat\beta^2}{h}
 \le\frac3{256}<\frac13.
\end{equation}
Thus $\widehat X\succ0$ and $X\succ0$. Also $\norm u\le2\sqrt m$ and
$\norm c=\sqrt m$, so
\[
 \sym(D+uc^{\T})\succeq(h-1-2m)I_{2m}\succ0.
\]
Therefore, by Lemma~\ref{lem:certificate}\textup{(ii)}, with the above choice of $K$, $M$ is Hurwitz.

Lastly, we note the encoding size of the construction.
The displayed block assemblies produce integer matrices with $O(n)$
nonzero entries. Their magnitudes are at most
$\lambda_m=4096m^2=256N^2$ in $A$ and at most $6$ in $B$.
There are $\Theta(n^2)$ entries in the dense input and only $O(n)$
nonzero entries, each with $O(\log(n+1))$ bits. Consequently
$L_n=\Theta(n^2)$, and the construction takes polynomial time.
\end{proof}

\begin{theorem}[Discrete time]\label{thm:bits-DT}
%
For each integer $n\ge1$, matrices
$A_{\rm d}\in\mathbb Q^{N\times N}$, $B_{\rm d}\in\mathbb Q^{N\times p}$, and a nonzero
row vector $C_{\rm d}\in\mathbb Q^{1\times N}$ can be constructed in time polynomial in $n$,
with
\[
 N=8n+12,\qquad p=n+2,
\]
and total size $L_{n,\rm d}=\Theta(n^2)$,
such that a rational stabilizing controller exists, but every (real)
controller $K_{\rm d}=(y_0,\ldots,y_n,s)^{\T}$ for which $A_{\rm d}+B_{\rm d}K_{\rm d}C_{\rm d}$ is Schur satisfies
\begin{equation}\label{eq:bits-small-entry2}
 0<y_n\le\frac23\,2^{-2^n}.
\end{equation}
In particular, writing $y_n=a/b$ in lowest terms with $b>0$ for any rational
stabilizing controller gives $b\ge\frac32\,2^{2^n}$; this denominator alone
requires more than $2^n$ bits.
\end{theorem}

\begin{proof}
First, we describe the construction and use a correspondence between continuous and discrete time.
To show that our construction admits a stabilizing controller, we will 
use the continuous-time matrices of Theorem~\ref{thm:bits-CT} and the
continuous-to-discrete transformation from Section~\ref{sec:dt}, with
parameter one.
For $i=1,\dots,m$, let $a_i=1+\lambda_i$ and define
\begin{equation}\label{eq:bits-Q}
 \begin{gathered}
 Q_0=\diag_{1\le i\le m}\!\left(\frac{I_2}{a_i}+\frac{J}{a_i^2}\right)
     =(I_{2m}+D)^{-1},\\
 \delta=1+c^{\T}Q_0u^0=1-\frac2{a_m^2}>0,\qquad
 Q=Q_0-\frac{Q_0u^0c^{\T}Q_0}{\delta}.
 \end{gathered}
\end{equation}
The Sherman–Morrison formula gives
$Q=(I_{2m}+D+u^0c^{\T})^{-1}$. Block multiplication then gives
\begin{equation}\label{eq:bits-R}
 R:=(I_N-A)^{-1}=\begin{pmatrix}Q&Q-I_{2m}\\Q&Q\end{pmatrix},
 \qquad c^{\T}Q=\frac{c^{\T}Q_0}{\delta}.
\end{equation}
Set
\begin{equation}\label{eq:bits-discrete-data}
 A_{\rm d}=2R-I_N,\qquad B_{\rm d}=RB,\qquad
 C_{\rm d}=2CR,\qquad f=CRB\in\mathbb Q^{1\times p}.
\end{equation}
Note that $f$ is a row vector and thus $fK=CRBK$ is a scalar for each controller
$K\in\mathbb R^{p\times1}$.

For $M=A+BKC$ and $M_{\rm d}=A_{\rm d}+B_{\rm d}K_{\rm d}C_{\rm d}$,
the matrix determinant lemma gives
\[
 \begin{aligned}
 \det(I_N-M)&=\det(I_N-A)(1-fK),\\
 \det(I_N+M_{\rm d})&=2^N\det(R)(1+fK_{\rm d}).
 \end{aligned}
\]
If $M$ is Hurwitz, then $I_N-M$ is invertible; if $M_{\rm d}$ is Schur,
then $I_N+M_{\rm d}$ is invertible. 
Thus the denominators $1 - fK$ and $1 + fK_{\rm d}$ in the inverse
controller maps
\begin{equation}\label{eq:bits-controller-map}
 K_{\rm d}=\frac{K}{1-fK},\qquad
 K=\frac{K_{\rm d}}{1+fK_{\rm d}}
\end{equation}
do not vanish at stabilizing controllers. The Sherman–Morrison formula yields
\begin{equation}\label{eq:bits-loop-map}
 M_{\rm d}=2(I_N-M)^{-1}-I_N=(I_N+M)(I_N-M)^{-1}.
\end{equation}
For any symmetric matrix $P\in\mathbb R^{N\times N}$, we have
\begin{equation}\label{eq:bits-Lyapunov-map}
 M_{\rm d}^{\T}PM_{\rm d}-P
 =2(I_N-M)^{-\T}(M^{\T}P+PM)(I_N-M)^{-1}.
\end{equation}
Lemma~\ref{lem:cayley} shows that
\eqref{eq:bits-controller-map} is a bijection between stabilizing controllers.
It preserves rationality, so the controller~\eqref{eq:bits-witness} gives a
rational Schur stabilizer.

Second, we show that the entry $y_n$ of any Schur-stabilizing controller satisfies the claimed bound. For every continuous-time stabilizer,~\eqref{eq:bits-chain} and
\eqref{eq:bits-normalized-chain} give $t>0$ and $0\le x_j\le t/6$.
Let $u_i^0\in\mathbb R^2$ be block $i$ of $u^0$, and set
$w_i=u_i-u_i^0+v_i$. 
From the definitions of $u_i$ and $v_i$ in~\eqref{eq:bits-vectors} it follows that
$(w_i)_1\ge-t$ and $(w_i)_2\ge t$ for $1\le i\le m$.
Equations~\eqref{eq:bits-Q}--\eqref{eq:bits-discrete-data} imply
\begin{equation}\label{eq:bits-gain-sign}
 fK =-\frac1\delta\sum_{i=1}^m
 \left(\frac{(w_i)_1}{a_i^2}+\frac{(w_i)_2}{a_i}\right)<0,
\end{equation}

because each summand is at least $t(a_i^{-1}-a_i^{-2})>0$.
Thus $1-fK>1$. Every Schur-stabilizing controller corresponds to such a
continuous-time controller, and therefore
\[
 0<y_n=\frac{x_n}{1-fK}\le x_n\le\frac23\,2^{-2^n}.
\]
The denominator conclusion follows from~\eqref{eq:bits-elementary}.

Lastly, we note the encoding size of the construction.
Each row and column of $Q_0$ has at most two nonzero entries, all with
$O(\log(n+1))$ bits. Since $Q_0u^0$ has only two nonzero entries,
\eqref{eq:bits-Q}--\eqref{eq:bits-discrete-data} show that $Q,R,A_{\rm d}$,
and $C_{\rm d}$ also have $O(n)$ nonzero entries, each with
$O(\log(n+1))$ bits.
For $W=U+V\in\mathbb Z^{2m\times p}$, we have
\[
 B_{\rm d}=\binom{V-QW}{-QW},\qquad
 QW=Q_0W-\frac{(Q_0u^0)(c^{\T}Q_0W)}{\delta}.
\]
Every column of $W$ except its last has at most four nonzero entries. Thus $B_{\rm d}$
has $O(n)$ nonzero entries, each of $O(\log(n+1))$ bits, except possibly
four entries in its last column. Those entries involve sums of $O(n)$
rational terms and have $O(n\log(n+1))$ bits. Including all zero entries,
the input size is therefore $L_{n,\rm d}=\Theta(n^2)$.
\end{proof}



\begin{thebibliography}{7}












\bibitem{BlondelGeversLindquist}
V.~Blondel, M.~Gevers, and A.~Lindquist,
Survey on the state of systems and control,
\emph{European Journal of Control}, \textbf{1}(1) (1995), 5--23.

\bibitem{BlondelTsitsiklis}
V.~Blondel and J.~N. Tsitsiklis,
NP-hardness of some linear control design problems,
\emph{SIAM Journal on Control and Optimization}, \textbf{35}(6) (1997), 2118--2127.

\bibitem{BoydEtAl}
S.~Boyd, L.~El~Ghaoui, E.~Feron, and V.~Balakrishnan,
\emph{Linear Matrix Inequalities in System and Control Theory},
SIAM Studies in Applied Mathematics, vol.~15,
Society for Industrial and Applied Mathematics, Philadelphia, 1994.

\bibitem{Chaudhry}
A.~Chaudhry,
\emph{Algebraic Methods in Convex Geometry, Nonlinear Optimization, and Learning Dynamical Systems},
Ph.D. thesis, Princeton University, 2024.
\url{https://dataspace.princeton.edu/handle/88435/dsp01vt150n65n}.


\bibitem{GareyJohnson}
M.~R. Garey and D.~S. Johnson,
\emph{Computers and Intractability: A Guide to the Theory of NP-Completeness},
W.~H. Freeman and Company, San Francisco, 1979.

\bibitem{grigor1988solving}
D.~Yu. Grigor'ev and N.~N. Vorobjov, Jr.,
Solving systems of polynomial inequalities in subexponential time,
\emph{Journal of Symbolic Computation}, \textbf{5} (1988).

\bibitem{Staffans}
O.~J. Staffans,
\emph{Well-Posed Linear Systems},
Encyclopedia of Mathematics and its Applications, vol.~103,
Cambridge University Press, Cambridge, 2005.

\bibitem{Wonham}
W.~M. Wonham,
\emph{Linear Multivariable Control: A Geometric Approach},
Springer-Verlag, New York, 1979.


\bibitem{Bellman}
R. Bellman,
\emph{Introduction to Matrix Analysis},
McGraw-Hill, New York, 1960.

\end{thebibliography}
\end{document}